\documentclass[11pt,reqno]{amsart}
\usepackage[T1]{fontenc}
\usepackage{lmodern}
\usepackage{amsmath,amssymb,amsthm,mathtools}
\usepackage[margin=1in]{geometry}
\usepackage[hidelinks]{hyperref}
\theoremstyle{plain}
\newtheorem{theorem}{Theorem}[section]
\newtheorem{lemma}[theorem]{Lemma}
\newtheorem{proposition}[theorem]{Proposition}
\newtheorem{corollary}[theorem]{Corollary}
\theoremstyle{definition}

\newtheorem{example}[theorem]{Example}
\theoremstyle{remark}
\newtheorem{remark}[theorem]{Remark}

\newcommand{\K}{K}
\newcommand{\Nov}{\mathrm{Nov}}
\newcommand{\Lie}{\mathrm{Lie}}
\newcommand{\sgn}{\mathrm{sgn}}
\newcommand{\Span}{\operatorname{span}}
\newcommand{\Hilb}{\operatorname{Hilb}}
\newcommand{\Ind}{\operatorname{Ind}}
\newcommand{\Res}{\operatorname{Res}}
\newcommand{\Hom}{\operatorname{Hom}}
\newcommand{\Sym}{\operatorname{Sym}}

\newcommand{\bZ}{\mathbb Z}
\newcommand{\bC}{\mathbb C}
\newcommand{\bQ}{\mathbb Q}
\newcommand{\bfe}{\mathbf 1}
\newcommand{\cD}{\mathcal D}
\newcommand{\cH}{\mathcal H}

\title[Lie elements in a free Novikov algebra]{A vanishing criterion for Lie elements in a free Novikov algebra}

\author{F. A. Mashurov}
\author{B. K. Sartayev}

\thanks{The results presented in this work were obtained with the assistance of the artificial intelligence system ChatGPT Astra. The mathematical arguments, computations, and proofs were subsequently independently checked and additionally verified using Claude Fable.}


\begin{document}

\begin{abstract}
Every Novikov algebra is Lie-admissible: the commutator turns it into a Lie algebra.
We give a finite criterion for a multilinear element of the free Novikov algebra to belong to the Lie subalgebra generated by the free generators.
By the differential realization of free Novikov algebras, the multilinear component of degree $n$ is identified with the space of homogeneous polynomials of degree $n-1$ in $n$ variables.
We prove that a multilinear element is a Lie element if and only if its symbol vanishes at every integer point $(a_1,\ldots,a_n)$ with $a_i\le 1$ and $a_1+\cdots+a_n\ge 2$.
Equivalently, the symbol is annihilated by two explicit linear differential operators of orders two and three.
The proof combines the Witt algebra, a specialization argument for a symmetric block of variables, homogeneous interpolation and Molev's description of the multilinear component of the Lie algebra generated by the free generators as a module over the symmetric group.
As applications we show that a nonzero multilinear Lie element is never a total derivative, recover the dimension of the multilinear Lie component, and illustrate the criterion by examples in degrees $4$, $5$ and $6$.
\end{abstract}

\maketitle

\section{Introduction}

Let $\K$ be a field of characteristic $0$.
An algebra $(N,\circ)$ over $\K$ is called a \emph{Novikov algebra} if it satisfies the identities
\begin{align}
(a\circ b)\circ c-a\circ(b\circ c)
&=
(b\circ a)\circ c-b\circ(a\circ c),
\label{left-sym}\\
(a\circ b)\circ c
&=
(a\circ c)\circ b.
\label{right-comm}
\end{align}
The identity \eqref{left-sym} is the left-symmetric identity, and \eqref{right-comm} is the right-commutative identity.
Algebras of this type appeared in the work of Gelfand and Dorfman \cite{GD} on Hamiltonian operators in the formal calculus of variations, and in the work of Balinskii and Novikov \cite{BN} on Poisson brackets of hydrodynamic type.
The name ``Novikov algebra'' was proposed by Osborn \cite{Os}.
The structure theory of Novikov algebras was developed by Zelmanov \cite{Ze}, Osborn \cite{Os}, Filippov \cite{Fi} and Xu \cite{Xu}; see also \cite{SZ} for recent results on solvability and nilpotency.

The basic example of a Novikov algebra is the following one.
Let $A$ be a commutative associative algebra with a derivation $D$.
Then $A$ with the multiplication
\begin{equation}
a\circ b=aD(b)
\label{GD-product}
\end{equation}
is a Novikov algebra \cite{GD}.
Free Novikov algebras were described by Dzhumadil'daev and L\"ofwall \cite{DL}: the free Novikov algebra generated by a set $X$ is isomorphic to the subalgebra generated by $X$ in the free commutative differential algebra $\K\{X\}$ with the product \eqref{GD-product}.
In particular, a basis of the free Novikov algebra is given by differential monomials, and the multilinear component of degree $n$ has dimension $\binom{2n-2}{n-1}$.
The $S_n$- and $GL_n$-module structures of free Novikov algebras were studied in \cite{DI14}, Gr\"obner--Shirshov bases in \cite{BCZ}, and the codimension growth of the Novikov operad in \cite{Dz11}.
For the Jordan part of Novikov algebras we refer to \cite{Dz02}; for the connection of Novikov algebras with Manin products of operads and derived identities see \cite{KSO}.

Every Novikov algebra is Lie-admissible, that is, the commutator
\begin{equation}
[a,b]=a\circ b-b\circ a
\label{commutator}
\end{equation}
satisfies the Jacobi identity.
Thus for the free Novikov algebra $\Nov\langle X\rangle$ one may consider the Lie subalgebra $\Lie\langle X\rangle$ of the commutator algebra $\Nov\langle X\rangle^{(-)}$ generated by the free generators $X$.
Its elements are called \emph{Lie elements} of $\Nov\langle X\rangle$.

The problem of recognizing Lie elements inside a free algebra is classical.
For the free associative algebra it is solved by the Friedrichs criterion \cite{Fr} and by the Dynkin--Specht--Wever theorem, see \cite{Re}.
Analogous questions for free bicommutative algebras were studied in \cite{DI18, DIT}.
For free Novikov algebras the problem was raised by Dzhumadil'daev and Ismailov \cite{DI14}.
Note that the Lie subalgebra $\Lie\langle X\rangle$ is not free: already in degree $5$ its multilinear component has dimension $20<4!$ (see Corollary \ref{dimension}), so it satisfies identities that do not follow from the Jacobi identity.
For the related problem of polynomial identities of Novikov algebras and of the algebras derived from them we refer to \cite{Dz11, KSO}.
Therefore the description of Lie elements cannot be reduced to the free Lie algebra, and one needs a criterion adapted to the Novikov structure.

The realization of $\Nov\langle X\rangle$ inside $\K\{X\}$ has the following consequence.
For $X=\{x_1,\ldots,x_n\}$, the multilinear component $\Nov(n)$ of $\Nov\langle X\rangle$ is spanned by the differential monomials
\[
x_1^{(\alpha_1)}\cdots x_n^{(\alpha_n)},
\qquad
\alpha_1+\cdots+\alpha_n=n-1,
\]
where $x^{(r)}=D^r(x)$.
Replacing $x_i^{(\alpha_i)}$ by $t_i^{\alpha_i}$ identifies $\Nov(n)$ with the space $S_{n-1}$ of homogeneous polynomials of degree $n-1$ in $t_1,\ldots,t_n$.
The image $P_f\in S_{n-1}$ of $f\in\Nov(n)$ is called the \emph{symbol} of $f$.
Under this identification, the symbol of the commutator $[f,g]$ of elements in disjoint sets of variables is the product of $P_f$, $P_g$ and a linear form; hence symbols of Lie monomials are products of explicit linear forms, and the space $L_n\subseteq S_{n-1}$ of symbols of multilinear Lie elements is spanned by such products.
The recognition of Lie elements is thereby reduced to the description of a concrete subspace of a polynomial ring.

The aim of this paper is to give a finite criterion for membership in $L_n$.
For $n\ge2$, put
\begin{equation}
V_n
=
\{a=(a_1,\ldots,a_n)\in\bZ^n
\mid
a_i\le1\ (1\le i\le n),\
a_1+\cdots+a_n\ge2\}.
\label{Vn-intro}
\end{equation}
The set $V_n$ is finite; it consists of $\binom{2n-2}{n}$ points.

\begin{theorem}
\label{main-intro}
Let $\K$ be a field of characteristic $0$ and $f\in\Nov(n)$, $n\ge2$.
Then $f$ is a Lie element if and only if
\[
P_f(a)=0
\qquad\text{for every }a\in V_n.
\]
\end{theorem}

The vanishing condition can be replaced by two linear differential equations.

\begin{theorem}
\label{pde-intro}
Let $P\in S_{n-1}$, $n\ge2$.
Then $P$ is the symbol of a multilinear Lie element of $\Nov\langle X\rangle$ if and only if
\begin{align}
\sum_{i=1}^{n}
\Bigl(
t_i\frac{\partial^2P}{\partial t_i^2}
-2\frac{\partial P}{\partial t_i}
\Bigr)
&=0,
\label{PDE1-intro}\\
\sum_{i=1}^{n}
\Bigl(
t_i\frac{\partial^3P}{\partial t_i^3}
-3\frac{\partial^2P}{\partial t_i^2}
\Bigr)
&=0.
\label{PDE2-intro}
\end{align}
\end{theorem}

Neither of the equations \eqref{PDE1-intro}, \eqref{PDE2-intro} alone is sufficient; see Examples \ref{ex4} and \ref{ex6}.

The necessity of the vanishing condition comes from the Witt algebra $W_1=\Span\{e_r\mid r\in\bZ\}$, $[e_r,e_s]=(s-r)e_{r+s}$, which is the commutator algebra of the Laurent polynomial ring with the product \eqref{GD-product}, $D=z\,d/dz$.
The elements $e_r$ with $r\le1$ span a Lie subalgebra, and the evaluation of a multilinear Lie element at $e_{a_1},\ldots,e_{a_n}$ equals $P_f(a)e_{a_1+\cdots+a_n}$.
The sufficiency is the main part of the paper.
We first show that the vanishing condition is equivalent to the system of differential equations $\cD_kP=0$, $k\ge1$, where
\[
\cD_k=\sum_{i=1}^{n}\bigl(t_i\partial_i^{k+1}-(k+1)\partial_i^k\bigr),
\]
and that the operators $\cD_k$ satisfy the Witt relations $[\cD_k,\cD_l]=(k-l)\cD_{k+l}$; this explains why two equations suffice.
The differential form of the condition yields an injectivity property of the specialization $t_1=\cdots=t_a=0$ on polynomials symmetric in $t_1,\ldots,t_a$ (Lemma \ref{specialization}).
Combined with homogeneous interpolation, this gives an upper bound for the dimension of every mixed symmetry component of the space of polynomials satisfying the vanishing condition (Lemma \ref{mixed-bound}).
The bound coincides with the dimension of the corresponding component of the degree $n-1$ part of the coinvariant algebra of $S_n$, which is isomorphic to $L_n$ by a theorem of Molev \cite{Mo}.
Since every irreducible $S_n$-module is detected by some mixed symmetry component (Lemma \ref{detection}), the reverse inclusion follows.

As applications of the criterion we prove that a nonzero multilinear Lie element is never a total derivative (Proposition \ref{no-total-derivative}), that the Novikov product $f\circ g$ of two nonzero Lie elements in disjoint variables is never a Lie element (Remark \ref{product-remark}), and we recover the formula
\[
\dim L_n
=
[z^{n-1}]\prod_{j=1}^{n}(1+z+\cdots+z^{j-1}),
\]
that is, $\dim L_n$ is the number of permutations of $n$ letters with exactly $n-1$ inversions.
In the last section the criterion is illustrated by examples in degrees $4$, $5$ and $6$.

Throughout the paper $\K$ is a field of characteristic $0$, $[n]=\{1,\ldots,n\}$, $S=\K[t_1,\ldots,t_n]$ and $S_d$ is the space of homogeneous polynomials of degree $d$ in $S$.
For a finite set $A\subseteq[n]$ we write $|A|$ for its cardinality, and for $b\in\bZ_{\ge0}^n$ we write $|b|=b_1+\cdots+b_n$.

\section{Free Novikov algebras and symbols}

Let $(N,\circ)$ be a Novikov algebra, i.e.\ an algebra satisfying \eqref{left-sym} and \eqref{right-comm}.
It is well known that $N$ is Lie-admissible: the commutator \eqref{commutator} satisfies the Jacobi identity.
Indeed, \eqref{left-sym} means that the associator $(a,b,c)=(a\circ b)\circ c-a\circ(b\circ c)$ is symmetric in $a,b$, and for any algebra with this property the Jacobi identity for the commutator follows from the standard expansion of $[[a,b],c]+[[b,c],a]+[[c,a],b]$ as a sum of associators.
The Lie algebra $(N,[\cdot,\cdot])$ is denoted by $N^{(-)}$.

If $A$ is a commutative associative algebra with a derivation $D$, then $(A,\circ)$ with $a\circ b=aD(b)$ satisfies \eqref{left-sym} and \eqref{right-comm}:
\[
(a\circ b)\circ c-a\circ(b\circ c)=-\,abD^2(c),
\qquad
(a\circ b)\circ c=aD(b)D(c).
\]

Let $X=\{x_1,\ldots,x_n\}$ and let
\begin{equation}
\K\{X\}
=
\K[x_i^{(r)}\mid 1\le i\le n,\ r\ge0]
\label{diff-alg}
\end{equation}
be the free commutative differential algebra generated by $X$, where $x_i=x_i^{(0)}$ and the derivation $D$ is defined by $D(x_i^{(r)})=x_i^{(r+1)}$.
With the product $f\circ g=fD(g)$ the algebra $\K\{X\}$ is a Novikov algebra.
By \cite{DL}, the subalgebra of $(\K\{X\},\circ)$ generated by $X$ is the free Novikov algebra $\Nov\langle X\rangle$ (see also \cite{DI14, BCZ}).

Let $\Nov(n)$ be the multilinear component of $\Nov\langle X\rangle$ in $x_1,\ldots,x_n$.
By \cite{DL}, it has the basis
\begin{equation}
x^{(\alpha)}
=
x_1^{(\alpha_1)}\cdots x_n^{(\alpha_n)},
\qquad
\alpha=(\alpha_1,\ldots,\alpha_n)\in\bZ_{\ge0}^n,\quad
|\alpha|=n-1.
\label{Nov-base}
\end{equation}
Hence the linear map
\begin{equation}
x^{(\alpha)}\longmapsto t^{\alpha}=t_1^{\alpha_1}\cdots t_n^{\alpha_n}
\label{symbol-map}
\end{equation}
identifies $\Nov(n)$ with $S_{n-1}$.
For $f\in\Nov(n)$ we denote its image under \eqref{symbol-map} by $P_f$ and call it the \emph{symbol} of $f$.

More generally, for $A\subseteq[n]$ let $\K\{X\}_A$ be the span of the differential monomials $\prod_{i\in A}x_i^{(\alpha_i)}$, $\alpha_i\ge0$, which are multilinear in the variables indexed by $A$.
The symbol map is extended to $\K\{X\}_A$ by
\begin{equation}
P_{\prod_{i\in A}x_i^{(\alpha_i)}}
=
\prod_{i\in A}t_i^{\alpha_i}
\label{symbol-ext}
\end{equation}
and linearity.
For $A\subseteq[n]$ put
\begin{equation}
\sigma_A=\sum_{i\in A}t_i,
\qquad
\sigma=\sigma_{[n]}=t_1+\cdots+t_n.
\label{sigma}
\end{equation}

\begin{lemma}
\label{symbol-rules}
Let $A,B\subseteq[n]$ be disjoint, $f\in\K\{X\}_A$ and $g\in\K\{X\}_B$.
Then
\begin{align}
P_{f\circ g}&=P_fP_g\,\sigma_B,
\label{symbol-circ}\\
P_{[f,g]}&=P_fP_g\,(\sigma_B-\sigma_A),
\label{symbol-bracket}\\
P_{D(f)}&=P_f\,\sigma_A.
\label{symbol-D}
\end{align}
\end{lemma}

\begin{proof}
By linearity we may assume that $g=\prod_{i\in B}x_i^{(\beta_i)}$.
By the Leibniz rule,
\[
D(g)
=
\sum_{j\in B}x_j^{(\beta_j+1)}\prod_{i\in B,\ i\ne j}x_i^{(\beta_i)},
\]
and applying \eqref{symbol-ext} we get $P_{D(g)}=P_g\sum_{j\in B}t_j=P_g\sigma_B$.
This proves \eqref{symbol-D}, and \eqref{symbol-circ} follows since $f\circ g=fD(g)$ and the symbol of a product of monomials in disjoint variables is the product of symbols.
Finally, \eqref{symbol-bracket} is obtained by subtracting $P_{g\circ f}=P_fP_g\sigma_A$ from \eqref{symbol-circ}.
\end{proof}

An element $f\in\Nov\langle X\rangle$ is called a \emph{Lie element} if it belongs to the Lie subalgebra $\Lie\langle X\rangle$ of $\Nov\langle X\rangle^{(-)}$ generated by $X$.
Let
\begin{equation}
L_n=\{P_f\mid f\in\Nov(n)\text{ is a Lie element}\}\subseteq S_{n-1}
\label{Ln}
\end{equation}
be the space of symbols of multilinear Lie elements of degree $n$.
The symmetric group $S_n$ acts on $\Nov(n)$ by permutations of the variables and on $S_{n-1}$ by permutations of $t_1,\ldots,t_n$; the symbol map is $S_n$-equivariant and $L_n$ is an $S_n$-submodule of $S_{n-1}$.

Let $\tau$ be a binary tree with leaves labelled by $1,\ldots,n$, and let $L_v,R_v$ be the sets of labels of the left and right subtrees at an internal vertex $v$ of $\tau$.
Then $\tau$ defines a Lie monomial (a nested commutator) in $x_1,\ldots,x_n$, and by \eqref{symbol-bracket} its symbol is
\begin{equation}
P_\tau=\prod_{v}(\sigma_{R_v}-\sigma_{L_v}),
\label{tree-symbol}
\end{equation}
the product being taken over the internal vertices of $\tau$.
Consequently $L_n=\Span\{P_\tau\}$.

For $\pi\in S_n$, put
\[
[x_{\pi(1)},\ldots,x_{\pi(n)}]
=
[\cdots[[x_{\pi(1)},x_{\pi(2)}],x_{\pi(3)}],\ldots,x_{\pi(n)}]
\]
(left-normed commutator). By \eqref{tree-symbol}, its symbol is
\begin{equation}\label{comb}
r_\pi
=
\prod_{j=2}^{n}
\Bigl(t_{\pi(j)}-\sum_{i<j}t_{\pi(i)}\Bigr).
\end{equation}
Since the multilinear component of a free Lie algebra is spanned by left-normed commutators \cite{Re}, we obtain
\begin{equation}
L_n=\Span_{\K}\{r_\pi\mid\pi\in S_n\}.
\label{Ln-span}
\end{equation}

For example, $r_{\mathrm{id}}=(t_2-t_1)(t_3-t_1-t_2)=t_1^2-t_2^2-t_1t_3+t_2t_3$ for $n=3$, i.e.
\[
[[x_1,x_2],x_3]
=
x_1^{(2)}x_2x_3-x_1x_2^{(2)}x_3-x_1^{(1)}x_2x_3^{(1)}+x_1x_2^{(1)}x_3^{(1)}.
\]

\section{The vanishing set and the main theorem}

For $n\ge2$ define
\begin{equation}
V_n
=
\bigl\{
(1-b_1,\ldots,1-b_n)
\mid
b\in\bZ_{\ge0}^n,\ |b|\le n-2
\bigr\}.
\label{Vn}
\end{equation}
Equivalently, $V_n$ is the set \eqref{Vn-intro} of integer vectors $a$ with $a_i\le1$ and $a_1+\cdots+a_n\ge2$.
Put
\begin{equation}
K_n
=
\{P\in S_{n-1}\mid P(a)=0\text{ for every }a\in V_n\}.
\label{Kn}
\end{equation}
The set $V_n$ is $S_n$-invariant, hence $K_n$ is an $S_n$-submodule of $S_{n-1}$.

\begin{theorem}
\label{main-thm}
Let $\K$ be a field of characteristic $0$ and $n\ge2$.
Then
\begin{equation}
L_n=K_n.
\label{L=K}
\end{equation}
In other words, $f\in\Nov(n)$ is a Lie element if and only if $P_f(a)=0$ for every $a\in V_n$.
\end{theorem}

Theorem \ref{main-thm} is Theorem \ref{main-intro} of the introduction.

Both spaces in \eqref{L=K} are defined over $\bZ$: $L_n$ is spanned by the polynomials \eqref{comb}, which have integer coefficients, and $K_n$ is the kernel of the evaluation map $S_{n-1}\to\K^{V_n}$ at the integer points of $V_n$, whose matrix in the monomial basis has integer entries.
Therefore $\dim_{\K}L_n$ and $\dim_{\K}K_n$ do not depend on the field $\K$ of characteristic $0$, and the inclusion $L_n\subseteq K_n$ (once proved over $\bQ$) is preserved by extension of scalars.
Consequently, it suffices to prove Theorem \ref{main-thm} for $\K=\bC$.
From now on until the end of Section \ref{sec-proof} we assume $\K=\bC$.

Let $e_1,\ldots,e_n$ be the elementary symmetric polynomials in $t_1,\ldots,t_n$. The coinvariant algebra
\begin{equation}
\cH_n=\bC[t_1,\ldots,t_n]/(e_1,\ldots,e_n)
\label{coinv}
\end{equation}
is a graded $S_n$-module; it is isomorphic to the graded $S_n$-module of harmonic polynomials and, as an ungraded module, to the regular representation \cite{Ch, St}.
We denote by $(\cH_n)_d$ its homogeneous component of degree $d$.

The following two results are due to Molev \cite{Mo}; we state them in the form used below.

\begin{theorem}[Molev]
\label{Molev-module}
As $S_n$-modules, $L_n\simeq(\cH_n)_{n-1}$.
\end{theorem}

For integers $q\ge m\ge2$, put
\begin{equation}
E_{q,m}
=
\bigl\{
(1-u_1,\ldots,1-u_q)
\mid
u\in\bZ_{\ge0}^q,\ |u|\le m-2
\bigr\},
\label{Eqm}
\end{equation}
and $E_{q,m}=\varnothing$ for $m\le1$.
Note that $E_{n,n}=V_n$.

\begin{theorem}[Molev's homogeneous interpolation lemma]
\label{Molev-interpolation}
Let $q\ge m\ge2$ and $r\ge2m-4$.
Then every function $E_{q,m}\to\bC$ is the restriction of a homogeneous polynomial of degree $r$ in $q$ variables.
\end{theorem}

\section{Differential form of the vanishing condition}
\label{sec-diff}

We begin with an elementary interpolation lemma.

\begin{lemma}
\label{simplex}
Let $Q(u_1,\ldots,u_q)$ be a polynomial of total degree at most $d$ such that $Q(u)=0$ for every $u\in\bZ_{\ge0}^q$ with $|u|\le d$.
Then $Q=0$.
\end{lemma}

\begin{proof}
For $\alpha\in\bZ_{\ge0}^q$ put $\binom{u}{\alpha}=\prod_{i=1}^{q}\binom{u_i}{\alpha_i}$.
The polynomials $\binom{u}{\alpha}$, $|\alpha|\le d$, form a basis of the space of polynomials of total degree at most $d$.
Write $Q=\sum_{|\alpha|\le d}c_\alpha\binom{u}{\alpha}$.
Since $\binom{\beta}{\alpha}=0$ unless $\alpha\le\beta$ componentwise, and $\binom{\alpha}{\alpha}=1$, evaluation at $u=\beta$ in increasing order of $|\beta|$ gives a triangular system with unit diagonal for the coefficients $c_\alpha$.
Hence all $c_\alpha=0$.
\end{proof}

For $r\ge0$ define the polynomials
\begin{equation}
h_0(t)=1,
\qquad
h_r(t)=\prod_{j=0}^{r-1}(t+j-1)=(t-1)t(t+1)\cdots(t+r-2)
\quad(r\ge1),
\label{hr}
\end{equation}
and let $T$ be the linear operator on $S$ defined on monomials by
\begin{equation}
T(t_1^{\alpha_1}\cdots t_n^{\alpha_n})
=
h_{\alpha_1}(t_1)\cdots h_{\alpha_n}(t_n).
\label{T-def}
\end{equation}
Since $h_r(t)=t^r+(\text{terms of lower degree})$, the operator $T$ is unipotent with respect to the degree filtration; in particular $T$ is invertible and preserves the highest homogeneous component of every polynomial.

For $k\ge1$ put
\begin{equation}
\cD_k
=
\sum_{i=1}^{n}
\bigl(
t_i\partial_i^{k+1}-(k+1)\partial_i^{k}
\bigr),
\qquad
\partial_i=\frac{\partial}{\partial t_i}.
\label{Dk}
\end{equation}
Each $\cD_k$ is homogeneous of degree $-k$: it maps $S_d$ to $S_{d-k}$.
In one variable, for $r\ge k$,
\begin{equation}
\bigl(t\partial^{k+1}-(k+1)\partial^{k}\bigr)t^{r}
=
\frac{r!}{(r-k)!}\,(r-2k-1)\,t^{r-k},
\label{Dk-monomial}
\end{equation}
and the left-hand side is zero for $r<k$.

\begin{lemma}
\label{differential-criterion}
For $P\in S_{n-1}$, $n\ge2$, the following conditions are equivalent:
\begin{enumerate}
\item[(i)] $P\in K_n$;
\item[(ii)] $TP=P$;
\item[(iii)] $\cD_kP=0$ for every $k\ge1$;
\item[(iv)] $\cD_1P=\cD_2P=0$.
\end{enumerate}
\end{lemma}

\begin{proof}
Put $d=n-1$.

\emph{Step 1: $TP$ vanishes on $V_n$ for every $P\in S_d$.}
Let $|\alpha|=d$ and $b\in\bZ_{\ge0}^n$ with $|b|\le d-1$.
Then $b_i<\alpha_i$ for some $i$, and
\[
h_{\alpha_i}(1-b_i)=\prod_{j=0}^{\alpha_i-1}(j-b_i)=0,
\]
since the factor with $j=b_i$ vanishes.
Hence $T(t^\alpha)$ vanishes at $1-b$, and by linearity $(TP)(a)=0$ for all $a\in V_n$.

\emph{Step 2: (i)$\Leftrightarrow$(ii).}
Let $P\in K_n$.
By Step 1 the polynomial $P-TP$ vanishes on $V_n$, and it has total degree at most $d-1$ because $T$ preserves the highest homogeneous component.
After the substitution $t_i=1-b_i$ it becomes a polynomial in $b$ of degree at most $d-1$ vanishing at all $b\in\bZ_{\ge0}^n$ with $|b|\le d-1$.
By Lemma \ref{simplex}, $P-TP=0$.
Conversely, if $TP=P$ then $P$ vanishes on $V_n$ by Step 1.

\emph{Step 3: a conjugation formula.}
Let $E=\sum_{i=1}^{n}t_i\partial_i$ be the Euler operator.
We claim that
\begin{equation}
T^{-1}ET
=
E-\sum_{k\ge1}\frac{\cD_k}{k(k+1)}.
\label{conjugation}
\end{equation}
Since $T$, $E$ and $\cD_k$ are sums of operators acting in a single variable, it suffices to verify \eqref{conjugation} in one variable $t$.
The exponential generating function of the polynomials \eqref{hr} is
\begin{equation}
T(e^{tz})
=
\sum_{r\ge0}h_r(t)\frac{z^r}{r!}
=
(1-z)^{1-t},
\label{genfun}
\end{equation}
as follows from the binomial series $(1-z)^{1-t}=\sum_r\binom{1-t}{r}(-z)^r$.
Differentiating \eqref{genfun} with respect to $t$ and $z$ gives
\[
\partial_t\bigl((1-z)^{1-t}\bigr)=-\log(1-z)\,(1-z)^{1-t},
\qquad
t(1-z)^{1-t}=\bigl((1-z)\partial_z+1\bigr)(1-z)^{1-t}.
\]
Since $\partial_z e^{tz}=te^{tz}$ and $z e^{tz}=\partial_t e^{tz}$, these identities mean
\begin{equation}
T^{-1}\partial T=-\log(1-\partial),
\qquad
T^{-1}tT=t(1-\partial)+1,
\label{conj-basic}
\end{equation}
where $\partial=\partial_t$.
Consequently
\begin{align*}
T^{-1}(t\partial)T
&=
\bigl(t(1-\partial)+1\bigr)\sum_{k\ge1}\frac{\partial^k}{k}\\
&=
t\partial
+\sum_{k\ge1}t\partial^{k+1}\Bigl(\frac1{k+1}-\frac1k\Bigr)
+\sum_{k\ge1}\frac{\partial^k}{k}\\
&=
t\partial-\sum_{k\ge1}\frac{t\partial^{k+1}-(k+1)\partial^{k}}{k(k+1)}.
\end{align*}
Summation over the variables proves \eqref{conjugation}.

\emph{Step 4: (ii)$\Leftrightarrow$(iii).}
Assume $TP=P$.
Since $P$ is homogeneous of degree $d$, $EP=dP$, and \eqref{conjugation} applied to $P=T^{-1}(TP)$ gives
\[
\sum_{k\ge1}\frac{\cD_kP}{k(k+1)}
=
EP-T^{-1}ETP
=
dP-T^{-1}EP
=
dP-d\,T^{-1}P
=
0,
\]
where we used $TP=P$, hence $T^{-1}P=P$.
The summands $\cD_kP$ have pairwise distinct degrees $d-k$, hence $\cD_kP=0$ for all $k\ge1$.
Conversely, if all $\cD_kP$ vanish, then by \eqref{conjugation} $T^{-1}ETP=EP=dP$, i.e.\ $E(TP)=d\,TP$.
Thus $TP$ is homogeneous of degree $d$; since its highest homogeneous component is $P$, we get $TP=P$.

\emph{Step 5: (iii)$\Leftrightarrow$(iv).}
A direct computation in one variable, followed by summation over $i$, gives
\begin{equation}
[\cD_k,\cD_l]=(k-l)\cD_{k+l},
\qquad k,l\ge1.
\label{Witt-relations}
\end{equation}
In particular $[\cD_1,\cD_k]=(1-k)\cD_{k+1}$, and $\cD_1P=\cD_2P=0$ implies successively $\cD_3P=\cD_4P=\cdots=0$.
\end{proof}

\begin{remark}
By \eqref{Witt-relations}, the operators $\cD_k$, $k\ge1$, span a Lie algebra isomorphic to the positive part of the Witt algebra.
Condition (iv) says that $P$ is annihilated by the two generators $\cD_1$, $\cD_2$ of this Lie algebra.
\end{remark}

\section{Specialization of a symmetric block of variables}
\label{sec-spec}

For $1\le a\le n$ let $S_a\subseteq S_n$ act on the first $a$ variables $t_1,\ldots,t_a$, and put
\begin{equation}
K_n^{S_a}
=
\{P\in K_n\mid P\text{ is symmetric in }t_1,\ldots,t_a\}.
\label{KnSa}
\end{equation}
Define the specialization map
\begin{equation}
\rho_a(P)
=
P(\underbrace{0,\ldots,0}_{a},t_{a+1},\ldots,t_n).
\label{rho}
\end{equation}

\begin{lemma}
\label{specialization}
The map $\rho_a\colon K_n^{S_a}\to\bC[t_{a+1},\ldots,t_n]$ is injective.
Moreover, if $q=n-a\ge2$, then $\rho_a(P)$ is a homogeneous polynomial of degree $n-1$ vanishing on $V_q$ for every $P\in K_n^{S_a}$.
\end{lemma}

\begin{proof}
The second assertion is immediate.
If $(1-b_1,\ldots,1-b_q)\in V_q$, then $|b|\le q-2$, so
\[
(\underbrace{0,\ldots,0}_{a},1-b_1,\ldots,1-b_q)
=
(\underbrace{1-1,\ldots,1-1}_{a},1-b_1,\ldots,1-b_q)
\in V_n,
\]
since $a+|b|\le a+q-2=n-2$.
Hence $\rho_a(P)$ vanishes on $V_q$.

We prove injectivity.
Put $d=n-1$ and write $P=\sum_{|\alpha|=d}c_\alpha t^\alpha$.
Introduce independent variables $u_i^{(r)}$, $1\le i\le n$, $0\le r\le d$, (``jet variables'') and put
\begin{equation}
F_P
=
\sum_{\alpha}c_\alpha\,u_1^{(\alpha_1)}\cdots u_n^{(\alpha_n)}.
\label{FP}
\end{equation}
Thus $F_P$ is the differential polynomial whose symbol is $P$, written in the variables $u_i^{(r)}$; it is multilinear in the $n$ families $u_i^{(\cdot)}$.

For $1\le k\le d$ define a derivation $\delta_k$ of the polynomial ring in the jet variables by
\begin{equation}
\delta_k\bigl(u_i^{(r)}\bigr)
=
\begin{cases}
\Bigl(\dbinom{r}{k+1}-\dbinom{r}{k}\Bigr)u_i^{(r-k)},
& r\ge k,\\[3mm]
0,& r<k.
\end{cases}
\label{delta-def}
\end{equation}
By \eqref{Dk-monomial},
\[
\frac1{(k+1)!}\bigl(t\partial^{k+1}-(k+1)\partial^k\bigr)t^r
=
\frac{r!\,(r-2k-1)}{(k+1)!\,(r-k)!}\,t^{r-k}
=
\Bigl(\binom{r}{k+1}-\binom{r}{k}\Bigr)t^{r-k},
\]
and therefore, for every $P\in S$,
\begin{equation}
\delta_k(F_P)=\frac1{(k+1)!}F_{\cD_kP}.
\label{delta-D}
\end{equation}
Every $\delta_k$ is locally nilpotent, since it strictly lowers the total jet order $\sum r$ of a monomial.
Hence the exponentials $\exp(z\delta_k)$, $z\in\bC$, are well-defined automorphisms (polynomial flows) of the ring of jet variables.

If $P\in K_n$, then $\cD_kP=0$ for all $k$ by Lemma \ref{differential-criterion}, and \eqref{delta-D} gives $\delta_k(F_P)=0$.
Thus $F_P$ is invariant under all flows $\exp(z\delta_k)$.

Now let $P\in K_n^{S_a}$ with $\rho_a(P)=0$.
Identify the first $a$ jet families $u_1^{(r)},\ldots,u_a^{(r)}$ with a single family $u^{(r)}$, $0\le r\le d$, and let $G$ be the polynomial obtained from $F_P$ by this identification.
The identification commutes with every $\delta_k$, because \eqref{delta-def} acts on each family in the same way; hence $\delta_k(G)=0$ and $G$ is invariant under all flows $\exp(z\delta_k)$.

Take an arbitrary point of the jet space with $u^{(0)}\ne0$.
By \eqref{delta-def},
\[
\delta_k\bigl(u^{(k)}\bigr)=-u^{(0)},
\qquad
\delta_k\bigl(u^{(r)}\bigr)=0\quad(r<k),
\]
so that $\exp(z\delta_k)$ sends $u^{(k)}$ to $u^{(k)}-zu^{(0)}$ and leaves $u^{(0)},\ldots,u^{(k-1)}$ unchanged.
Applying successively $\exp(z_1\delta_1),\ldots,\exp(z_d\delta_d)$ with $z_k=u^{(k)}/u^{(0)}$ (the value of $u^{(k)}$ taken at the $k$-th step), we reach a point at which
\[
u^{(1)}=\cdots=u^{(d)}=0 .
\]
The flows may change the higher jets and the remaining jet families, but this causes no difficulty.

Since $F_P$ is multilinear in the first $a$ families, $G$ is homogeneous of degree $a$ in the variables $u^{(r)}$, and the only monomials of $G$ that survive at a point with $u^{(1)}=\cdots=u^{(d)}=0$ are those containing $(u^{(0)})^a$.
Their coefficients are exactly the coefficients of $P$ at the monomials with $\alpha_1=\cdots=\alpha_a=0$, i.e.\ the coefficients of $\rho_a(P)$.
Hence the value of $G$ at the resulting point equals
\[
(u^{(0)})^a\,F_{\rho_a(P)}(\text{remaining jet families})=0 .
\]
Since $G$ is invariant under the flows used, $G$ vanishes at every point with $u^{(0)}\ne0$, which is a Zariski dense set.
Thus $G=0$.

Finally, $F_P$ is symmetric and multilinear in its first $a$ families, and $G$ is its diagonalization.
By polarization, $a!\,F_P$ is recovered from $G$, so $F_P=0$ and $P=0$.
\end{proof}

\section{Mixed symmetry components}
\label{sec-mixed}

For $b\ge0$ put
\begin{equation}
H_b(z)=\prod_{j=1}^{b}\frac1{1-z^j},
\qquad
H_0(z)=1,
\qquad
p_b(r)=[z^r]H_b(z),
\label{Hb}
\end{equation}
so that $p_b(r)$ is the number of partitions of $r$ into at most $b$ parts, i.e.\ the dimension of the space of symmetric homogeneous polynomials of degree $r$ in $b$ variables.
We put $p_b(r)=0$ for $r<0$.

Let $\lambda\vdash n$.
Remove from the Young diagram of $\lambda$ all columns of length at least $3$; let their lengths be $c_1,\ldots,c_\ell$ ($c_i\ge3$), and let the remaining two rows have lengths $a\ge b\ge0$.
Put
\begin{equation}
C=c_1+\cdots+c_\ell,
\qquad
\delta=\sum_{i=1}^{\ell}\binom{c_i}{2},
\qquad
n=a+b+C,
\qquad
\delta\ge C.
\label{lambda-data}
\end{equation}
Choose disjoint blocks of variables of sizes $a,b,c_1,\ldots,c_\ell$ and put
\begin{equation}
G_\lambda
=
S_a\times S_b\times S_{c_1}\times\cdots\times S_{c_\ell}\subseteq S_n,
\qquad
\chi_\lambda
=
\bfe\boxtimes\bfe\boxtimes\sgn\boxtimes\cdots\boxtimes\sgn.
\label{G-chi}
\end{equation}
For an $S_n$-module $M$ put
\begin{equation}
M^{\chi}
=
\{v\in M\mid gv=\chi_\lambda(g)v\text{ for all }g\in G_\lambda\}.
\label{chi-subspace}
\end{equation}
Thus the elements of $M^{\chi}$ are symmetric in each of the first two blocks and alternating in each of the column blocks.
We call $M^\chi$ the \emph{mixed symmetry component} of $M$ associated with $\lambda$.

\begin{lemma}
\label{mixed-bound}
With the notation \eqref{lambda-data}--\eqref{chi-subspace},
\begin{equation}
\dim K_n^{\chi}
\le
[z^{n-1}]
\Bigl(
z^{\delta}\prod_{j=1}^{n}(1-z^j)\,
H_a(z)H_b(z)\prod_{i=1}^{\ell}H_{c_i}(z)
\Bigr).
\label{mixed-estimate}
\end{equation}
\end{lemma}

\begin{proof}
\emph{The case $a=0$.}
Then $b=0$ and $C=n$.
A polynomial alternating in a block of $c_i$ variables is divisible by the Vandermonde polynomial of that block, of degree $\binom{c_i}{2}$.
Hence every element of $K_n^\chi$ is divisible by a polynomial of degree $\delta\ge C=n>n-1$, and $K_n^\chi=0$.
The right-hand side of \eqref{mixed-estimate} is zero as well, since $\delta\ge n$.

\emph{The case $a\ge1$.}
Put $d=n-1$ and $q=n-a=b+C$.
By Lemma \ref{specialization}, setting the variables of the $a$-block equal to zero defines an injective linear map from $K_n^\chi$ to the space of homogeneous polynomials $Q=Q(w,v)$ of degree $d$ in the remaining $q$ variables, where $w=(w_1,\ldots,w_b)$ are the variables of the second symmetric block and $v=(v_1,\ldots,v_C)$ are the variables of the column blocks.
Each such $Q$ is symmetric in $w$, alternating in each column block, and, if $q\ge2$, vanishes on $V_q$.

Filter the space of these polynomials by the total degree in $v$.
Let $Q\ne0$ and let $s$ be the largest $v$-degree occurring in $Q$.
The component of $v$-degree $s$ can be written as
\begin{equation}
Q_s(w,v)=\sum_jA_j(v)B_j(w),
\label{leading-column}
\end{equation}
where the $A_j$ are linearly independent homogeneous polynomials of degree $s$ in $v$, alternating in each column block, and the $B_j$ are symmetric homogeneous polynomials of degree
\begin{equation}
r=d-s
\label{r-def}
\end{equation}
in $w$.
Since $A_j$ is divisible by the product of the Vandermonde polynomials of the column blocks, $s\ge\delta\ge C$.
Put
\begin{equation}
m=q-s=b+C-s\le b .
\label{m-def}
\end{equation}

\emph{Vanishing of the $B_j$.}
Suppose $m\ge2$ and $C\ge1$.
Take $w=(1-u_1,\ldots,1-u_b)\in E_{b,m}$, so that $|u|\le m-2$.
For every $v'\in\bZ_{\ge0}^C$ with $|v'|\le s$ we have $|u|+|v'|\le m-2+s=q-2$, hence $(w,1-v')\in V_q$ and $Q(w,1-v')=0$.
For fixed $w$, the polynomial $Q(w,v)$ has degree at most $s$ in $v$; by Lemma \ref{simplex}, applied after the substitution $v=1-v'$, it is zero as a polynomial in $v$.
In particular its component \eqref{leading-column} of $v$-degree $s$ is zero, and the linear independence of the $A_j$ gives $B_j(w)=0$ for all $j$.
Thus
\begin{equation}
B_j\big|_{E_{b,m}}=0\qquad\text{for all }j.
\label{Bj-vanish}
\end{equation}
If $C=0$, then $s=0$, $m=q=b$, and \eqref{Bj-vanish} follows directly by evaluating $Q=Q(w)$ on $V_b=E_{b,b}$.
If $m\le1$, then $E_{b,m}=\varnothing$ and \eqref{Bj-vanish} is an empty condition.

\emph{Counting the $B_j$.}
By \eqref{r-def} and \eqref{m-def},
\begin{equation}
r=d-s=n-1-(q-m)=a+m-1.
\label{r-relation}
\end{equation}
If $m\ge2$, then $a\ge b\ge m$ by \eqref{m-def}, so $r=a+m-1\ge2m-1\ge2m-4$, and Theorem \ref{Molev-interpolation} applies to $E_{b,m}$ and degree $r$.
The $S_b$-orbits on $E_{b,m}$ correspond to the partitions of the integers $0\le j\le m-2$ into at most $b$ parts; hence the space of $S_b$-invariant functions on $E_{b,m}$ has dimension $\sum_{j=0}^{m-2}p_b(j)$.
By Theorem \ref{Molev-interpolation} every function on $E_{b,m}$ is the restriction of a homogeneous polynomial of degree $r$, and averaging over $S_b$ shows that every $S_b$-invariant function is the restriction of a symmetric homogeneous polynomial of degree $r$.
Therefore the restriction map from $\Sym_r(w)$ (symmetric homogeneous polynomials of degree $r$ in $w$) to $S_b$-invariant functions on $E_{b,m}$ is surjective, and the space of $B\in\Sym_r(w)$ satisfying $B|_{E_{b,m}}=0$ has dimension
\begin{equation}
p_b(r)-\sum_{j=0}^{m-2}p_b(j)
=
p_b(r)-\sum_{j=0}^{r-a-1}p_b(j),
\label{kernel-dim}
\end{equation}
where we used $m-2=r-a-1$ from \eqref{r-relation}.
For $m\le1$ the sum is empty and \eqref{kernel-dim} equals $p_b(r)$; this agrees with the absence of vanishing conditions in that case.

\emph{A generating function.}
Put
\begin{equation}
U_a(z)=\prod_{j=a+1}^{\infty}(1-z^j).
\label{Ua}
\end{equation}
Since $r=a+m-1\le a+b-1\le2a-1$, no product of two or more nonconstant terms of $U_a(z)$ contributes to the coefficient of $z^r$ in $H_b(z)U_a(z)$.
Therefore
\begin{equation}
[z^r]\,H_b(z)U_a(z)
=
p_b(r)-\sum_{j=a+1}^{r}p_b(r-j)
=
p_b(r)-\sum_{j=0}^{r-a-1}p_b(j),
\label{HbUa}
\end{equation}
which is \eqref{kernel-dim}.

\emph{Conclusion.}
The Hilbert series of the space of polynomials in $v$ alternating in each column block is
\begin{equation}
z^{\delta}\prod_{i=1}^{\ell}H_{c_i}(z),
\label{alt-Hilbert}
\end{equation}
because such a polynomial is the product of the Vandermonde polynomials of the blocks, of total degree $\delta$, and a polynomial symmetric in each block.
Passing to the associated graded space with respect to the $v$-degree, we see that $\dim K_n^\chi$ is bounded by the sum over $s$ of the products of $[z^s]$ of \eqref{alt-Hilbert} and \eqref{kernel-dim} with $r=d-s$; by \eqref{HbUa} this sum is
\begin{equation}
\dim K_n^\chi
\le
[z^{d}]
\Bigl(
z^{\delta}\prod_{i=1}^{\ell}H_{c_i}(z)\,H_b(z)U_a(z)
\Bigr).
\label{before-U}
\end{equation}
Finally, $U_a(z)=H_a(z)\prod_{j=1}^{\infty}(1-z^j)$, and since $d=n-1$ the factors $(1-z^j)$ with $j>n$ do not affect the coefficient of $z^d$.
Hence \eqref{before-U} is \eqref{mixed-estimate}.
\end{proof}

\begin{lemma}
\label{detection}
Let $\lambda\vdash n$ and let $G_\lambda$, $\chi_\lambda$ be as in \eqref{G-chi}.
If $S^\lambda$ is the irreducible $S_n$-module (Specht module) corresponding to $\lambda$, then $(S^\lambda)^\chi\ne0$.
\end{lemma}

\begin{proof}
By Frobenius reciprocity,
\[
\dim(S^\lambda)^\chi
=
\dim\Hom_{G_\lambda}\bigl(\chi_\lambda,\Res^{S_n}_{G_\lambda}S^\lambda\bigr)
=
\dim\Hom_{S_n}\bigl(\Ind_{G_\lambda}^{S_n}\chi_\lambda,S^\lambda\bigr),
\]
and $\Ind_{G_\lambda}^{S_n}\chi_\lambda$ corresponds, under the characteristic map, to the product $h_ah_be_{c_1}\cdots e_{c_\ell}$ of complete and elementary symmetric functions \cite{Mac}.
By the Pieri rules, $s_\lambda$ occurs in this product with positive multiplicity: starting from the empty diagram, add the vertical strips of lengths $c_1,\ldots,c_\ell$ (this gives the diagram formed by the columns of $\lambda$ of length at least $3$), then add a horizontal strip of length $a$ to the first row and a horizontal strip of length $b$ to the second row.
The result is the diagram of $\lambda$.
Hence $(S^\lambda)^\chi\ne0$.
\end{proof}

\section{Proof of Theorem \ref{main-thm}}
\label{sec-proof}

\subsection{The inclusion \texorpdfstring{$L_n\subseteq K_n$}{Ln subseteq Kn}}

Consider the Laurent polynomial algebra $A=\bC[z,z^{-1}]$ with the derivation $D=z\,\frac{d}{dz}$ and the Novikov product $f\circ g=fD(g)$.
Put $e_r=z^r$, $r\in\bZ$.
Then $D(e_r)=re_r$, $e_r\circ e_s=se_{r+s}$ and
\begin{equation}
[e_r,e_s]=(s-r)e_{r+s},
\label{Witt}
\end{equation}
so $A^{(-)}$ is the Witt algebra.
Let
\begin{equation}
W_{\le1}=\Span\{e_r\mid r\le1\}.
\label{W1}
\end{equation}
This is a Lie subalgebra of $A^{(-)}$: if $r,s\le1$ and $r\ne s$, then one of $r,s$ is at most $0$, so $r+s\le1$; if $r=s$ the bracket \eqref{Witt} is zero.

Let $f\in\Nov(n)$ be a Lie element and $a\in V_n$.
Since $a_i\le1$, the Lie subalgebra of $A^{(-)}$ generated by $e_{a_1},\ldots,e_{a_n}$ is contained in $W_{\le1}$; in particular $f(e_{a_1},\ldots,e_{a_n})\in W_{\le1}$.
On the other hand, the substitution $x_i\mapsto e_{a_i}$ sends $x_i^{(\alpha_i)}$ to $a_i^{\alpha_i}e_{a_i}$, hence
\begin{equation}
f(e_{a_1},\ldots,e_{a_n})=P_f(a)\,e_{a_1+\cdots+a_n}.
\label{Witt-evaluation}
\end{equation}
Since $a_1+\cdots+a_n\ge2$, the element $e_{a_1+\cdots+a_n}$ does not belong to $W_{\le1}$.
Therefore \eqref{Witt-evaluation} lies in $W_{\le1}$ only if $P_f(a)=0$.
Thus $L_n\subseteq K_n$.

\subsection{The inclusion \texorpdfstring{$K_n\subseteq L_n$}{Kn subseteq Ln}}
Fix $\lambda\vdash n$ and let $a,b,c_1,\ldots,c_\ell$, $G_\lambda$, $\chi_\lambda$ be as in Section \ref{sec-mixed}.

The Hilbert series of the mixed symmetry component $S^\chi$ of the polynomial ring is
\begin{equation}
\Hilb(S^\chi,z)
=
z^{\delta}H_a(z)H_b(z)\prod_{i=1}^{\ell}H_{c_i}(z):
\label{Schi-Hilbert}
\end{equation}
symmetry in a block of $r$ variables contributes the factor $H_r(z)$, and alternation in a block of $c_i$ variables contributes $z^{\binom{c_i}{2}}H_{c_i}(z)$ by \eqref{alt-Hilbert}.

The elementary symmetric polynomials $e_1,\ldots,e_n$ form a regular sequence in $S$ of degrees $1,\ldots,n$, and each $e_j$ is $S_n$-invariant.
Hence the Koszul resolution of $\cH_n=S/(e_1,\ldots,e_n)$ is an exact sequence of graded $S_n$-modules, and since the functor $M\mapsto M^\chi$ is exact (the group $G_\lambda$ is finite and the characteristic is $0$), we obtain
\begin{equation}
\Hilb(\cH_n^\chi,z)
=
\prod_{j=1}^{n}(1-z^j)\,\Hilb(S^\chi,z)
=
z^{\delta}\prod_{j=1}^{n}(1-z^j)\,H_a(z)H_b(z)\prod_{i=1}^{\ell}H_{c_i}(z).
\label{Hchi-Hilbert}
\end{equation}
Comparing \eqref{Hchi-Hilbert} with Lemma \ref{mixed-bound} gives
\[
\dim K_n^\chi\le\dim(\cH_n)_{n-1}^{\chi}.
\]
By Theorem \ref{Molev-module}, $\dim(\cH_n)_{n-1}^{\chi}=\dim L_n^\chi$, hence
\[
\dim K_n^\chi\le\dim L_n^\chi .
\]
Since $L_n\subseteq K_n$, the reverse inequality holds, and therefore
\begin{equation}
K_n^\chi=L_n^\chi
\qquad\text{for every }\lambda\vdash n.
\label{equal-chi}
\end{equation}

Put $M=K_n/L_n$.
The category of finite-dimensional $S_n$-modules over $\bC$ is semisimple, and $M\mapsto M^\chi$ is exact; hence \eqref{equal-chi} gives $M^\chi=0$ for every $\lambda\vdash n$.
If $M\ne0$, choose an irreducible submodule $S^\lambda\subseteq M$; then $(S^\lambda)^\chi\ne0$ by Lemma \ref{detection}, so $M^\chi\ne0$, a contradiction.
Thus $M=0$ and $K_n=L_n$ over $\bC$.
By the reduction of Section 3, \eqref{L=K} holds over every field of characteristic $0$.
This completes the proof of Theorem \ref{main-thm}.
\qed

\subsection{The differential criterion}

\begin{theorem}
\label{differential-cor}
Let $P\in S_{n-1}$, $n\ge2$.
Then $P$ is the symbol of a multilinear Lie element of $\Nov\langle X\rangle$ if and only if
\begin{equation}
\cD_1P=0
\qquad\text{and}\qquad
\cD_2P=0,
\label{PDE}
\end{equation}
where $\cD_1,\cD_2$ are the operators \eqref{Dk}.
Explicitly, \eqref{PDE} is the system \eqref{PDE1-intro}, \eqref{PDE2-intro}.
\end{theorem}

\begin{proof}
By Theorem \ref{main-thm}, $P\in L_n$ if and only if $P\in K_n$, and by Lemma \ref{differential-criterion}, $P\in K_n$ if and only if $\cD_1P=\cD_2P=0$.
\end{proof}

\begin{remark}
\label{finite-remark}
Theorem \ref{main-thm} and Theorem \ref{differential-cor} give two finite procedures for recognizing multilinear Lie elements: evaluation at the $\binom{2n-2}{n}$ points of $V_n$, or the computation of the two polynomials $\cD_1P$ and $\cD_2P$.
Neither procedure requires a basis of $L_n$.
The differential criterion is more economical: $\cD_1$ and $\cD_2$ act on the monomial basis by \eqref{Dk-monomial}, and the number of monomials of $S_{n-1}$ is $\binom{2n-2}{n-1}$, whereas $|V_n|=\binom{2n-2}{n}$ evaluations are needed for the vanishing criterion.
\end{remark}

\subsection{Lie elements are not total derivatives}
By \eqref{symbol-D}, multiplication of a symbol by $\sigma=t_1+\cdots+t_n$ corresponds to the derivation $D$ of $\K\{X\}$.

\begin{proposition}
\label{no-total-derivative}
For every $n\ge2$,
\[
L_n\cap\sigma S_{n-2}=0 .
\]
In particular, a nonzero multilinear Lie element of $\Nov\langle X\rangle$ is not a total derivative, i.e.\ it is not of the form $D(g)$ with $g\in\K\{X\}$.
\end{proposition}

\begin{proof}
Let $P=\sigma Q\in L_n$ with $Q\in S_{n-2}$.
By Theorem \ref{main-thm}, $P(a)=0$ for every $a\in V_n$.
Since $\sigma(a)=a_1+\cdots+a_n\ge2$ for $a\in V_n$, we get $Q(a)=0$ for all $a\in V_n$.
After the substitution $a=1-b$, the polynomial $Q(1-b)$ has degree at most $n-2$ and vanishes at every $b\in\bZ_{\ge0}^n$ with $|b|\le n-2$.
By Lemma \ref{simplex}, $Q=0$, hence $P=0$.
\end{proof}

The Hilbert series of the coinvariant algebra \eqref{coinv} is
\[
\Hilb(\cH_n,z)
=
\prod_{j=1}^{n}\frac{1-z^j}{1-z}
=
\prod_{j=1}^{n}(1+z+\cdots+z^{j-1}).
\]
Therefore Theorem \ref{Molev-module} gives the following formula.

\begin{corollary}
\label{dimension}
For every $n\ge1$,
\begin{equation}
\dim L_n
=
[z^{n-1}]\prod_{j=1}^{n}(1+z+\cdots+z^{j-1}).
\label{dim-formula}
\end{equation}
Equivalently, $\dim L_n$ is the number of permutations of $n$ letters having exactly $n-1$ inversions.
\end{corollary}

The first values are
\[
\dim L_n=1,\ 1,\ 2,\ 6,\ 20,\ 71,\ 259,\ 961
\qquad(n=1,\ldots,8).
\]
In particular $\dim L_n=(n-1)!$ for $n\le4$, while $\dim L_5=20<24=4!$; hence the Lie subalgebra $\Lie\langle X\rangle$ satisfies multilinear identities of degree $5$ which do not follow from the Jacobi identity, and it is not a free Lie algebra.

In the examples below we use the notation $x^{(r)}=D^r(x)$, so that $x_i^{(1)}=D(x_i)$, and we write $\sigma_A$ as in \eqref{sigma}.

\begin{example}[Degree $4$]\label{ex4}
Here $n=4$, the symbols are homogeneous cubic polynomials in
$t_1,t_2,t_3,t_4$, and
\[
V_4=\{1-b\mid b\in\mathbb Z_{\geq 0}^4,\ |b|\leq 2\},
\qquad
|V_4|=\binom{6}{4}=15.
\]

\textup{(a) Recognizing a Lie element.}
Consider the multilinear element
\[
\begin{aligned}
f={}&-x_1^{(3)}x_2x_3x_4
-x_1^{(2)}x_2^{(1)}x_3x_4
+x_1^{(2)}x_2x_3x_4^{(1)}
+x_1^{(1)}x_2^{(2)}x_3x_4\\
&+x_1^{(1)}x_2x_3^{(2)}x_4
-x_1^{(1)}x_2x_3^{(1)}x_4^{(1)}
+x_1x_2^{(3)}x_3x_4\\
&-x_1x_2^{(2)}x_3x_4^{(1)}
-x_1x_2^{(1)}x_3^{(2)}x_4
+x_1x_2^{(1)}x_3^{(1)}x_4^{(1)}.
\end{aligned}
\]
Its symbol is
\[
\begin{aligned}
P_f={}&-t_1^3-t_1^2t_2+t_1^2t_4+t_1t_2^2+t_1t_3^2-t_1t_3t_4\\
&+t_2^3-t_2^2t_4-t_2t_3^2+t_2t_3t_4.
\end{aligned}
\]
A direct computation gives
\[
D_1P_f=0,
\qquad
D_2P_f=0.
\]
Hence Theorem~\ref{differential-cor} shows that $f$ is a
Lie element.

In this example one can also identify the corresponding Lie
polynomial explicitly. Indeed,
\[
P_f
=(t_2-t_1)(t_3-t_1-t_2)(t_4-t_1-t_2-t_3).
\]
By \eqref{comb}, the right-hand side is the symbol of the left-normed
commutator
\[
[[[x_1,x_2],x_3],x_4].
\]
Since the symbol map is injective on the multilinear component,
\[
f=[[[x_1,x_2],x_3],x_4].
\]
Thus the criterion recognizes the Lie character of $f$ directly from
its expansion in the differential basis, without requiring a
presentation of $f$ as an iterated commutator.

Notice also that no individual differential monomial
$x^{(\alpha)}$, $|\alpha|=n-1$, is a Lie element for $n\geq2$.
Indeed, its symbol is $t^\alpha$, and
\[
t^\alpha(1,\ldots,1)=1,
\]
whereas $(1,\ldots,1)\in V_n.$

\medskip

\textup{(b) Both equations \eqref{PDE1-intro} and \eqref{PDE2-intro} are needed.}
Consider
\[
f=x_1^{(3)}x_2x_3x_4,
\qquad
P_f=t_1^3.
\]
By \eqref{Dk-monomial}, with $r=3$,
\[
D_1(t_1^3)
=t_1\cdot 6t_1-2\cdot 3t_1^2=0,
\]
whereas
\[
D_2(t_1^3)
=t_1\cdot 6-3\cdot 6t_1
=-12t_1\neq0.
\]
Thus $P_f$ satisfies \eqref{PDE1-intro} but not \eqref{PDE2-intro}, and $f$ is not a
Lie element. Hence the first differential equation alone is not
sufficient.

\medskip

\textup{(c) A Novikov product of commutators.}
Let
\[
u=[x_1,x_2],
\qquad
v=[x_3,x_4],
\qquad
f=u\circ v.
\]
By \eqref{symbol-circ} and \eqref{symbol-bracket},
\[
P_f=(t_2-t_1)(t_4-t_3)(t_3+t_4).
\]
Take
\[
a=(0,1,0,1)=1-(1,0,1,0)\in V_4.
\]
Then
\[
P_f(a)=1\cdot1\cdot1=1\neq0.
\]
Therefore $u\circ v$ is not a Lie element.

On the other hand,
\[
g=u\circ v-v\circ u=[u,v]
=[[x_1,x_2],[x_3,x_4]]
\]
has symbol
\[
P_g
=(t_2-t_1)(t_4-t_3)
 (t_3+t_4-t_1-t_2).
\]
We verify directly that $P_g$ vanishes on $V_4$. Let
$a=(a_1,a_2,a_3,a_4)\in V_4$. If $a_1=a_2$ or $a_3=a_4$, one of the
first two factors vanishes. Otherwise the pairs
$(a_1,a_2)$ and $(a_3,a_4)$ consist of distinct integers not exceeding
$1$. Hence
\[
a_1+a_2\leq1,
\qquad
a_3+a_4\leq1.
\]
Since
\[
a_1+a_2+a_3+a_4\geq2,
\]
both inequalities must be equalities:
\[
a_1+a_2=a_3+a_4=1.
\]
Thus the last factor also vanishes. Therefore
\[
P_g(a)=0
\qquad\text{for every }a\in V_4,
\]
and Theorem~\ref{main-intro} gives again that
\[
[[x_1,x_2],[x_3,x_4]]
\]
is a Lie element.
\end{example}

\begin{example}[degree $5$]
\label{ex5}
Here $n=5$, symbols are quartic forms, and $|V_5|=\binom{8}{5}=56$.

\smallskip
(a) \emph{Left-normed Novikov monomials span no Lie element.}
For $k\in[5]$ put
\[
y_k=x_k\prod_{j\ne k}x_j^{(1)}
=
(((x_k\circ x_{j_1})\circ x_{j_2})\circ x_{j_3})\circ x_{j_4},
\qquad
\{j_1,j_2,j_3,j_4\}=[5]\setminus\{k\}
\]
(the right-hand side does not depend on the order of $j_1,\ldots,j_4$ by \eqref{right-comm}).
Its symbol is $P_{y_k}=\prod_{j\ne k}t_j$.
Let $f=\sum_{k=1}^{5}c_ky_k$, $c_k\in\K$, so that
\[
P_f=\sum_{k=1}^{5}c_k\prod_{j\ne k}t_j .
\]
\emph{Vanishing criterion.}
Let $a^{(k)}$ be the point with $a^{(k)}_k=0$ and $a^{(k)}_j=1$ for $j\ne k$; it belongs to $V_5$ since $b=\epsilon_k$ and $|b|=1\le3$.
Then $P_f(a^{(k)})=c_k$.
Hence $f$ is a Lie element only if $c_1=\cdots=c_5=0$.

\emph{Differential criterion.}
Since $P_f$ is multilinear in $t_1,\ldots,t_5$, all second derivatives $\partial_i^2P_f$ vanish, so $\cD_2P_f=0$ automatically and
\[
\cD_1P_f=-2\sum_{i=1}^{5}\partial_iP_f
=
-2\sum_{i<k}(c_i+c_k)\prod_{j\notin\{i,k\}}t_j .
\]
Thus $\cD_1P_f=0$ if and only if $c_i+c_k=0$ for all $i\ne k$, which for $n\ge3$ forces $c_1=\cdots=c_5=0$.
Both criteria give the same answer: the elements $y_1,\ldots,y_5$ span no nonzero Lie element.
The same holds for the elements $x_k\prod_{j\ne k}x_j^{(1)}$ in every degree $n\ge3$.

\smallskip
(b) \emph{Determining a coefficient.}
Let $u=[[x_1,x_2],x_3]$, $v=[x_4,x_5]$ and consider the one-parameter family
\[
f_c=u\circ v+c\,v\circ u,
\qquad c\in\K .
\]
By \eqref{symbol-circ}, \eqref{comb},
\[
P_{f_c}
=
(t_2-t_1)(t_3-t_1-t_2)(t_5-t_4)\bigl(t_4+t_5+c\,(t_1+t_2+t_3)\bigr).
\]
Take the point $a=(0,1,0,0,1)=1-(1,0,1,1,0)\in V_5$ ($|b|=3$).
The factors take the values $1,\ -1,\ 1,\ 1+c$, so
\[
P_{f_c}(a)=-(1+c).
\]
By Theorem \ref{main-thm}, $f_c$ can be a Lie element only for $c=-1$; and indeed $f_{-1}=[u,v]=[[[x_1,x_2],x_3],[x_4,x_5]]$ is a Lie element.
Thus a single evaluation determines the only value of $c$ for which $f_c$ is a Lie element.
\end{example}

\begin{example}[degree $6$]
\label{ex6}
Here $n=6$, symbols are quintic forms, and $|V_6|=\binom{10}{6}=210$.

\smallskip
(a) \emph{A point of $V_n$ with a negative coordinate.}
Let $u=[x_1,x_2]$, $v=[[x_3,x_4],[x_5,x_6]]$ and $f=u\circ v$.
Then
\[
P_f
=
(t_2-t_1)\,\sigma_{\{3,4,5,6\}}\,(t_4-t_3)(t_6-t_5)(t_5+t_6-t_3-t_4).
\]
Any point $a\in V_6$ with $P_f(a)\ne0$ must satisfy $a_1\ne a_2$, $a_3\ne a_4$, $a_5\ne a_6$ and $a_3+a_4\ne a_5+a_6$; the points with all coordinates in $\{0,1\}$ do not satisfy these conditions simultaneously together with $a_1+\cdots+a_6\ge2$.
Take instead
\[
a=(0,1,1,0,1,-1)=1-(1,0,0,1,0,2)\in V_6
\qquad(|b|=4).
\]
The factors of $P_f$ take the values $1,\ 1,\ -1,\ -2,\ -1$, hence $P_f(a)=-2\ne0$, and $u\circ v$ is not a Lie element.
This example shows that the negative coordinates of the points of $V_n$ cannot be dispensed with.

\smallskip
(b) \emph{The second equation alone is not sufficient.}
Let $f=x_1^{(5)}x_2x_3x_4x_5x_6$, $P_f=t_1^5$.
By \eqref{Dk-monomial},
\[
\cD_1(t_1^5)=5\cdot(5-3)\,t_1^4=10t_1^4\ne0,
\qquad
\cD_2(t_1^5)=\frac{5!}{3!}\,(5-5)\,t_1^3=0 .
\]
Thus $P_f$ satisfies \eqref{PDE2-intro} but not \eqref{PDE1-intro}, and $f$ is not a Lie element.
Together with Example \ref{ex4}(b) this shows that the two equations of Theorem \ref{differential-cor} are independent.
More generally, \eqref{Dk-monomial} shows that $\cD_k$ annihilates $t_i^{2k+1}$ and no other power $t_i^r$ with $r\ge k$.
\end{example}

\begin{remark}
\label{product-remark}
Examples \ref{ex4}(c) and \ref{ex6}(a) are instances of the following general fact, which follows from Theorem \ref{main-thm}:
\emph{if $f$ and $g$ are nonzero multilinear Lie elements in disjoint sets of variables $A$ and $B$, $A\sqcup B=[n]$, then $f\circ g$ is not a Lie element.}
Indeed, put $p=|A|$.
The symbol $P_f$ is a homogeneous polynomial of degree $p-1$ in the variables $t_i$, $i\in A$, which vanishes on $V_p$ (for $p\ge2$) by Theorem \ref{main-thm}.
If $P_f$ also vanished at all integer points $a_A=1-b$ with $b\in\bZ_{\ge0}^A$, $|b|=p-1$, then $P_f(1-b)$ would vanish for all $b\in\bZ_{\ge0}^A$ with $|b|\le p-1$, and Lemma \ref{simplex} would give $P_f=0$.
Hence there is $a_A\in\bZ^A$ with $a_i\le1$, $\sum_{i\in A}a_i=1$ and $P_f(a_A)\ne0$ (for $p=1$ take $a_A=(1)$).
Similarly there is $a_B$ with $a_i\le1$, $\sum_{i\in B}a_i=1$ and $P_g(a_B)\ne0$.
The point $a=(a_A,a_B)$ belongs to $V_n$, and by \eqref{symbol-circ}
\[
P_{f\circ g}(a)=P_f(a_A)P_g(a_B)\,\sigma_B(a_B)=P_f(a_A)P_g(a_B)\ne0 .
\]
By Theorem \ref{main-thm}, $f\circ g\notin\Lie\langle X\rangle$.
In particular, for $n=2$ the element $x_1\circ x_2$ is not a Lie element, as expected.
\end{remark}

\end{document}